\documentclass[a4paper,12pt]{article}

\usepackage{amscd}
\usepackage{amssymb,amsfonts,amsmath,amsthm}
\usepackage{verbatim}

\usepackage{amsmath}
\usepackage{amsthm}
\usepackage{amssymb}

\usepackage{latexsym}
\usepackage{array}
\usepackage{enumerate}

\numberwithin{equation}{section}

\newcommand{\CC}{\mathbb{C}}

\newcommand{\RR}{\mathbb{R}}

\newcommand{\ZZ}{\mathbb{Z}}

\newcommand{\F}{\mathcal{F}}
\newcommand{\G}{\mathcal{G}}

\renewcommand{\SS}{\mathcal{S}}

\newcommand{\PP}{{\mathbb P}}

\renewcommand{\dim}{{\rm dim}}

\newcommand{\e}{\varepsilon}
\newcommand{\Spec}{{\rm Spec}}

\newcommand{\supp}{{\rm supp}}

\newcommand{\Int}{{\rm Int}}
\newcommand{\relint}{{\rm rel.int}}

\newcommand{\Db}{{\bf D}^{b}}
\newcommand{\Dbc}{{\bf D}_{c}^{b}}

\newcommand{\tl}[1]{\widetilde{#1}}

\newtheorem{theorem}{Theorem}[section]

\newtheorem{corollary}[theorem]{Corollary}
\newtheorem{lemma}[theorem]{Lemma}
\newtheorem{proposition}[theorem]{Proposition}

\theoremstyle{definition}
\newtheorem{definition}[theorem]{Definition}
\theoremstyle{remark}
\newtheorem{remark}[theorem]{\sc Remark}
\newtheorem{example}[theorem]{\sc Example}

\title{Bifurcation values of polynomial functions 
and perverse sheaves 
}

\author{ Kiyoshi TAKEUCHI 
\footnote{Institute of Mathematics, University  of 
Tsukuba, 1-1-1, Tennodai, 
Tsukuba, Ibaraki, 305-8571, Japan. 
E-mail: takemicro@nifty.com } }

\date{}

\begin{document}

\maketitle

\begin{abstract}
We characterize bifurcation values of 
polynomial functions by using the theory of 
perverse sheaves and their vanishing 
cycles. In particular, by introducing 
a method to compute the jumps of the Euler 
characteristics with compact support of 
their fibers, we confirm the conjecture of 
N{\'e}methi-Zaharia in many 
cases. 
\end{abstract}

\maketitle

\section{Introduction}\label{sec:1}

For a polynomial function $f \colon \CC^n 
\longrightarrow \CC$ it is well-known that 
there exists a finite 
subset $B \subset \CC$ 
such that the restriction
\begin{equation}
\CC^n \setminus f^{-1}(B) \longrightarrow 
\CC \setminus B
\end{equation}
of $f$ is a $C^{\infty}$ locally trivial fibration. 
We denote by $B_f$ the smallest 
subset $B \subset \CC$ satisfying this 
condition. Let ${\rm Sing} f \subset \CC^n$ 
be the set of the critical points of 
$f \colon \CC^n \longrightarrow \CC$. Then by 
the definition of $B_f$, obviously we have 
$f( {\rm Sing} f) \subset B_f$. 
The elements of $B_f$ are called
\emph{bifurcation values} of $f$. 
The determination of the bifurcation set 
$B_f \subset \CC$ is a fundamental problem and was 
studied by many mathematicians and from 
several viewpoints, e.g. \cite{Broughton}, 
\cite{C-D-T-T}, \cite{H-L}, \cite{H-N}, \cite{N-Z}, 
\cite{Nguyen}, \cite{P}, \cite{S-T-1}, 
\cite{Tibar-book} and \cite{Zaharia}. 
The essential difficulty consists in the 
fact that in general $f$ has a 
lot of singularities at infinity. 
Here we study $B_f$ via the Newton polyhedron of $f$. 
We denote by $\Gamma_{\infty}(f)$ the convex hull of 
the Newton polytope $NP(f)$ of $f$ 
and the origin in $\RR^n$. 
We call it the Newton polyhedron 
at infinity of $f$. Throughout this paper 
we assume that $\dim \Gamma_{\infty}(f)=n$. 
Recall that $f$ is said to be 
\emph{convenient} if $\Gamma_{\infty}(f)$ 
intersects the positive part 
of each coordinate axis. 
Kouchnirenko \cite{Kushnirenko} proved that if 
$f$ is convenient and non-degenerate 
at infinity (for the definition 
see Section \ref{sec:3}) then $B_f = 
f( {\rm Sing} f)$. However, 
in the non-convenient case, 
N\'emethi and Zaharia \cite{N-Z} showed 
that more bifurcation values may occur 
due to so-called ``bad faces''. Let us explain 
this phenomenon here and refer for details to 
Section \ref{sec:3}. 

\begin{definition}(\cite{T-T}) 
We say that a face $\gamma \prec \Gamma_{\infty}(f)$ 
is \emph{atypical} if $0 \in \gamma$, 
$\dim \gamma \geq 1$ and 
the cone $\sigma ( \gamma ) \subset \RR^n$ 
which corresponds it in the dual 
fan of $\Gamma_{\infty}(f)$ (for the definition 
see Section \ref{sec:3}) is not contained in the 
first quadrant $\RR^n_+ := \RR^n_{\geq 0}$ 
of $\RR^n$. 
\end{definition}
This definition is closely related to that of 
the bad faces of $NP(f-f(0))$ 
in N{\'e}methi-Zaharia \cite{N-Z}. 
See Section \ref{sec:3} for the details 
and examples. 
In this paper, we consider the case where $f$ 
is not convenient. Let $\gamma_1, \ldots, \gamma_m$ 
be the atypical faces of 
$\Gamma_{\infty}(f)$. As we see in 
Theorem \ref{BESS} below, 
in the generic case where 
$f$ is non-degenerate 
at infinity, the singularities at infinity 
of $f$ are produced 
only from $\gamma_i$. 
For $1 \leq i \leq m$ let 
$K_i= f_{\gamma_i}( 
{\rm Sing} f_{\gamma_i}) \subset \CC$ 
be the set of the critical values of the $\gamma_i$-part 
\begin{equation}
f_{\gamma_i}: T= (\CC^*)^n \longrightarrow \CC
\end{equation}
of $f$. Let us set 
\begin{equation}
K_f = f( {\rm Sing} f) \cup \{ f(0) \} \cup 
( \cup_{i=1}^m K_i). 
\end{equation}
Then N{\'e}methi-Zaharia \cite{N-Z} 
proved the following fundamental result. 

\begin{theorem}\label{BESS}
(N{\'e}methi-Zaharia \cite{N-Z}) 
Assume that $f$ is non-degenerate 
at infinity. Then we have 
$B_f \subset K_f$. 
\end{theorem}
Moreover they proved the equality $B_f=K_f$ for $n=2$ 
and conjectured its validity 
in higher dimensions. The essential 
problem is to prove the inverse inclusion 
$K_f \subset B_f$. This has been a long 
standing conjecture until now. 
Later Zaharia \cite{Zaharia} 
proved $K_f \setminus \{ f(0) \} \subset B_f$ 
for $n \geq 2$ 
under some additional assumptions. In particular, 
he assumed that $f$ has isolated singularities 
at infinity on a fixed smooth toric compactification 
of $\CC^n$. We can easily see that 
even if $f$ is non-degenerate 
at infinity this 
condition is not satisfied in general. 
See \eqref{NIS} in 
the proof of Theorem \ref{SMT-1} below. 
Namely his assumption is very strong and 
moreover depends on 
the choice of a particular 
smooth toric compactification 
of $\CC^n$. In this paper, we overcome 
this problem by introducing 
the following intrinsic definition. 
For $1 \leq i \leq m$ let 
$L_{\gamma_i} \simeq \RR^{\dim \gamma_i}$ 
be the linear subspace of $\RR^n$ spanned by 
$\gamma_i$ and set $T_{i}= 
\Spec ( \CC [L_{\gamma_i} \cap \ZZ^n]) \simeq 
( \CC^*)^{\dim \gamma_i}$. We regard 
$f_{\gamma_i}$ as a regular function on 
$T_{i}$. 

\begin{definition}\label{ISAI} 
We say that \emph{$f$ has isolated singularities 
at infinity over $b \in K_f \setminus 
[f( {\rm Sing} f) \cup \{ f(0) \} ]$} 
if for any $1 \leq i \leq m$ 
the hypersurface $f_{\gamma_i}^{-1} 
(b) \subset T_{i} 
\simeq (\CC^*)^{\dim \gamma_i}$ in $T_{i}$ 
has only isolated singular points.  
We simply say that \emph{$f$ has 
isolated singularities 
at infinity} if it is so 
over any value $b \in K_f \setminus 
[f( {\rm Sing} f) \cup \{ f(0) \} ]$.
\end{definition}
With this new definition at hand, 
by using also the 
more sophisticated machinary of vanishing 
cycle functors for constructible sheaves 
we can eventually 
work on a singular toric variety. 
Then we use the theory of perverse sheaves 
to improve Zaharia's result. 
In this way, we prove the inverse inclusion 
$K_f \setminus \{ f(0) \} \subset B_f$ and 
confirm the conjecture of \cite{N-Z} 
in many cases. In particular, for $n=3$ 
we obtain the following result. 

\begin{theorem}\label{MT-2} 
Let $f: \CC^3 \longrightarrow \CC$ be a 
non-degenerate polynomial at 
infinity such that $\dim \Gamma_{\infty}(f)=3$. 
Then, if $f$ has isolated singularities 
at infinity over $b \in K_f \setminus 
[f( {\rm Sing} f) \cup \{ f(0) \} ]$, 
we have $b \in B_f$. 
In particular, if $f$ has isolated singularities 
at infinity, we have 
$K_f \setminus \{ f(0) \} \subset B_f$.
\end{theorem}

For $n=3$ in the generic case, 
we thus confirm the conjecture of \cite{N-Z}. 
In fact, to prove Theorem \ref{MT-2} we 
show moreover that the Euler characteristics 
with compact support of 
the fibers of $f : \CC^n \longrightarrow 
\CC$ jump at the point $b$. 
The jump of Euler characteristics was 
used as a test for the bifurcation locus 
in case of ``isolated singularities at infinity'' 
(defined in various ways) 
in many other articles and from different 
points of view (see \cite{ALM-1}, \cite{ALM-2}, 
\cite{Broughton}, \cite{H-L}, 
\cite{H-N}, \cite{P}, \cite{S-T-1}, \cite{S}, 
\cite{T-2} etc.). 
To introduce our results in 
higher dimensions, we need also 
the following definition. 

\begin{definition}
We say that an atypical face 
$\gamma_i \prec \Gamma_{\infty}(f)$ 
is \emph{relatively simple} if the cone 
$\sigma_i := \sigma ( \gamma_i ) 
\subset \RR^n$ which corresponds to it in the dual 
fan of $\Gamma_{\infty}(f)$ is simplicial 
or satisfies the condition 
$\dim \sigma_i \leq 3$. 
\end{definition}
This condition implies that the constant sheaf on 
the affine toric variety associated to 
the cone $\sigma_i$ such that 
$\dim \sigma_i = n - \dim \gamma_i$ 
is perverse (up to some shift). 
If $\sigma_i$ is simplicial, then 
the affine toric variety associated to it is 
an orbifold and the perversity follows. 
If $\dim \sigma_i \leq 3$ 
we can show the corresponding 
perversity by a result of 
Fieseler \cite{F} on the intersection 
cohomology complexes of toric varieties. 
See Lemma \ref{PL} below. 
In higher dimensions, this perversity is 
essential in our proof of the 
the inverse inclusion 
$K_f \setminus \{ f(0) \} \subset B_f$. 
Note that if $\dim \gamma_i \geq n-3$ 
we have $\dim \sigma_i \leq 3$ and the 
atypical face $\gamma_i$ 
is relatively simple. In particular, if 
$n \leq 4$ this condition is always satisfied. 
Now we define a function 
$\chi_c: \CC \longrightarrow \ZZ$ on $\CC$ by 
\begin{equation}
\chi_c(t)= \sum_{j \in \ZZ} (-1)^j \dim 
H^j_c ( f^{-1}(t); \CC ) \qquad (t \in \CC ).  
\end{equation}
Let us fix a point $b \in K_f \setminus 
 [f( {\rm Sing} f) \cup \{ f(0) \} ] 
\subset \cup_{i=1}^m K_i$ 
and define the jump $E_f(b) \in \ZZ$ 
of the function $\chi_c$ at $b$ by 
\begin{equation}
E_f(b)=(-1)^{n-1} \left\{ \chi_c(b+ \e ) 
- \chi_c(b) \right\} \in \ZZ,  
\end{equation}
where $\e >0$ is sufficiently small. 
Recall that for a polytope $\Delta$ in $\RR^n$ 
its relative interior $\relint ( \Delta )$ is 
the interior of $\Delta$ in its affine 
span ${\rm Aff} ( \Delta ) 
\simeq \RR^{\dim \Delta}$ in $\RR^n$. 
Then we have the following result. 

\begin{theorem}\label{MT-1} 
Assume that $\dim \Gamma_{\infty}(f)=n$, 
$f$ is non-degenerate at infinity and 
has isolated singularities 
at infinity over $b \in K_f \setminus 
[f( {\rm Sing} f) \cup \{ f(0) \} ]$ 
and for any $1 \leq i \leq m$ such that $b \in K_i$ 
we have $\relint ( \gamma_i) 
\subset \Int ( \RR_+^n)$. 
Assume also that there exists $1 \leq i \leq m$ 
such that $b \in K_i$ and 
$\gamma_i \prec \Gamma_{\infty}(f)$ 
is relatively simple. 
Then we have $E_f(b)>0$ and hence $b \in B_f$. 
\end{theorem}

If $n=4$, all the atypical faces 
$\gamma_i$ are relatively simple 
and we obtain the following corollary. 

\begin{corollary}\label{TT-IC-1} 
Let $f: \CC^4 \longrightarrow \CC$ be a 
non-degenerate polynomial at 
infinity such that $\dim \Gamma_{\infty}(f)=4$. 
Then, if $f$ has isolated singularities 
at infinity over $b \in K_f \setminus 
[f( {\rm Sing} f) \cup \{ f(0) \} ]$ and for 
any $1 \leq i \leq m$ such that $b \in K_i$ 
the condition $\relint ( \gamma_i) 
\subset \Int ( \RR_+^4)$ is satisfied, then 
we have $E_f(b)>0$. 
In particular, if $f$ has isolated singularities 
at infinity and $\Gamma_{\infty}(f) \setminus \{ 0 \} 
\subset \Int ( \RR_+^4)$, we have 
$K_f \setminus \{ f(0) \} \subset B_f$.
\end{corollary}

For general $n \geq 2$ we have also the 
following corollary. 

\begin{corollary}\label{TT-IC-3} 
Assume that $\dim \Gamma_{\infty}(f)=n$ and 
$f$ is non-degenerate at infinity. Then, 
if moreover $f$ has isolated singularities 
at infinity, $\Gamma_{\infty}(f) \setminus \{ 0 \} 
\subset \Int ( \RR_+^n)$ and all 
the atypical faces $\gamma_i$ 
$(1 \leq i \leq m)$ are relatively 
simple, then we have 
$K_f \setminus \{ f(0) \} \subset B_f$.
\end{corollary}

Since $\gamma_i$ 
is relatively simple if $\dim \gamma_i \geq n-3$,  
Theorem \ref{MT-1} extends the 
result of Zaharia \cite{Zaharia}. 
Indeed, he assumed the much stronger condition that 
for any $1 \leq i \leq m$ such that $b \in K_i$ 
we have $\dim \gamma_i =n-1$ (which implies also 
$\relint ( \gamma_i) \subset \Int ( \RR_+^n)$). 
His assumption means that on a fixed smooth toric 
compactification of $\CC^n$ compatible with 
$\Gamma_{\infty}(f)$ the function $f$ has 
isolated singular points only on $T$-orbits 
at infinity of dimension $n-1$ over the 
point $b \in K_f \setminus 
[f( {\rm Sing} f) \cup \{ f(0) \} ]$. 
However under our weaker assumption, in 
the proof of Theorem \ref{MT-1} we encounter 
non-isolated singular points of $f$ 
at infinity on such 
a smooth compactification 
(see \eqref{NIS}). 
We overcome this difficulty 
by reducing the problem to the case of 
isolated singular points. To this end, 
we consider the direct image of the vanishing 
cycle of a constructible sheaf 
by a special morphism 
\begin{equation}
\pi : 
X= X_{\Sigma_C^{\prime}} \longrightarrow 
X_{\Sigma_C}
\end{equation}
of toric varieties. In this way, we 
can eventually work 
on the singular toric variety $X_{\Sigma_C}$ 
canonically associated to $\Gamma_{\infty}(f)$. 
This is the reason why we can employ our 
intrinsic definition in Definition \ref{ISAI}. 
Then, on $X_{\Sigma_C}$ the function $f$ has only 
isolated singular points at infinity 
(over the point $b \in K_f \setminus 
[f( {\rm Sing} f) \cup \{ f(0) \} ]$). 
Finally, to finish the proofs of Theorems 
\ref{MT-2} and \ref{MT-1}, we apply the theory of 
perverse sheaves and their vanishing 
cycles.  Here we use the perversity of 
the constant sheaf on 
the toric variety associated to 
the cone $\sigma_i$ to obtain 
the positivity $E_f(b)>0$. For the 
moment, it is not clear if we can further 
relax the assumption on $\sigma_i$ 
by using the very general formula 
for vanishing cycle sheaves in 
Massey \cite[Lemma 2.2]{M} etc. Note also that 
our condition $\relint ( \gamma_i) 
\subset \Int ( \RR_+^n)$ in Theorem \ref{MT-1} 
is equivalent to 
the one $\sigma_i \cap \RR_+^n = \{ 0 \}$. 
However in higher dimensions, there still remain  
some atypical faces for which this condition 
is not satisfied (see Example \ref{EXP} below). 
So it is desirable to relax the condition 
$\sigma_i \cap \RR_+^n = \{ 0 \}$. 
In this direction, we have only a 
partial answer in Theorem \ref{SMT-3} 
which extends Theorems \ref{MT-2} and \ref{MT-1} 
in a unified manner. We hope that we can drop 
some of the conditions in it 
in the future. 

\bigskip
\noindent{\bf Acknowledgement:} 
The author would like to 
express his hearty gratitude to Professor 
Mihai Tib\u ar for drawing our attention to this 
interesting problem. Several discussions with him were  
very useful. The author thanks him also for his 
encouragement during the preparation of 
this paper. Moreover he is very grateful to 
the referee for many valuable suggestions.

\section{Review on constructible and perverse sheaves}\label{sec:2}

In this section, we recall some results on 
constructible and perverse sheaves. 
In this paper, we essentially 
follow the terminology of 
\cite{Dimca}, \cite{H-T-T} and \cite{K-S}. 
For example, for a topological 
space $X$ we denote by $\Db(X)$ the 
derived category whose objects are 
bounded complexes of sheaves 
of $\CC_X$-modules on $X$. 
Denote by $\Dbc(X)$ the full 
subcategory of $\Db(X)$ consisting of 
constructible objects. 

\begin{definition}
Let $X$ be an algebraic variety 
over $\CC$. Then we say that a $\ZZ$-valued 
function $\psi \colon X \longrightarrow \ZZ$ 
on $X$ is \emph{constructible} if 
there exists a stratification 
$X=\bigsqcup_{\alpha} X_{\alpha}$ of $X$ such 
that $\psi |_{X_{\alpha}}$ is 
constant for any $\alpha$. We denote by 
$F_{\ZZ}(X)$ the abelian group 
of constructible functions on $X$.
\end{definition}

Let $\F \in \Dbc(X)$ be a constructible 
sheaf (complex of sheaves) on an algebraic 
variety $X$ over $\CC$. Then we can 
naturally associate to it a constructible function 
$\chi ( \F) \in F_{\ZZ} (X)$ on $X$ defined by 
\begin{equation}
\chi ( \F)(x)= \sum_{j \in \ZZ} (-1)^j \dim 
H^j( \F)_x  \qquad (x \in X). 
\end{equation} 
For a constructible function 
$\psi \colon X \longrightarrow \ZZ$, we take a 
stratification $X=\bigsqcup_{\alpha}X_{\alpha}$ 
of $X$ such that 
$\psi |_{X_{\alpha}}$ is constant 
for any $\alpha$ as above. We denote the 
Euler characteristic of 
$X_{\alpha}$ by $\chi(X_{\alpha})$. Then we set
\begin{equation}
\int_X \psi :=
\sum_{\alpha}\chi(X_{\alpha}) 
\cdot \psi (x_{\alpha}) \in \ZZ,
\end{equation}
where $x_{\alpha}$ is a reference 
point in $X_{\alpha}$. Then we can easily 
show that $\int_X \psi \in \ZZ$ does 
not depend on the choice of the 
stratification $X=\bigsqcup_{\alpha} 
X_{\alpha}$ of $X$. Hence we obtain a 
homomorphism 
\begin{equation}
\int_X \colon F_{\ZZ}(X) \longrightarrow \ZZ
\end{equation}
of abelian groups. 
For $\psi \in F_{\ZZ}(X)$, we call $\int_X \psi \in 
\ZZ$ the topological (Euler) 
integral of $\psi$ over $X$. 
More generally, to a morphism 
$f : X \longrightarrow Y$ 
of algebraic varieties over $\CC$ we can 
associate a homomorphism $\int_f : 
F_{\ZZ}(X) \longrightarrow F_{\ZZ}(Y)$ 
of abelian groups as follows. 
For $\psi \in F_{\ZZ}(X)$ we define 
$\int_f \psi \in F_{\ZZ}(Y)$ by 
\begin{equation}
\left( \int_f \psi \right) (y) 
= \int_{f^{-1}(y)} \psi \in \ZZ 
 \qquad (y \in Y). 
\end{equation} 
Then for any constructible 
sheaf $\F \in \Dbc(X)$ on $X$ 
we have the equality 
\begin{equation}
\int_f \chi ( \F )= \chi ( Rf_* ( \F )). 
\end{equation} 
Now we recall the following well-known 
property of Deligne's vanishing cycle functors. 
Let $X$ be an algebraic variety over $\CC$ and 
$f : X \longrightarrow \CC$ a non-constant 
regular function on $X$ and set 
$X_0= \{ x \in X \ | \ f(x)=0 \} \subset X$. 
Then we denote Deligne's vanishing cycle functor 
associated to $f$ by 
\begin{equation}
\varphi_f : \Dbc(X) \longrightarrow \Dbc(X_0) 
\end{equation} 
(see \cite[Section 4.2]{Dimca} and 
\cite[Section 8.6]{K-S} etc. for the details). 

\begin{proposition}\label{PDI} 
(cf. \cite[Proposition 4.2.11]{Dimca} 
and \cite[Exercise VIII.15]{K-S} etc.) 
Let $\pi : Y \longrightarrow X$ be a proper 
morphism of algebraic varieties over $\CC$ 
and $f : X \longrightarrow \CC$ a non-constant 
regular function on $X$. Set $g=f \circ \pi 
 : Y \longrightarrow \CC$, $X_0= \{ x \in X \ | \ 
f(x)=0 \}$ and $Y_0= \{ y \in Y \ | \ 
g(y)=0 \}$. Then for any $\G \in \Dbc(Y)$ 
we have an isomorphism 
\begin{equation}
\varphi_{f} ( R \pi_* \G ) \simeq 
R( \pi|_{Y_0})_* \varphi_{g} ( \G ), 
\end{equation}
where the morphism 
$\pi|_{Y_0} : Y_0 \longrightarrow X_0$ 
is induced by $\pi$. 
\end{proposition}

Recall that for an algebraic variety $X$ over 
$\CC$ the category ${\rm Perv}(X)$ of 
perverse sheaves on it is a full subcategory of 
$\Dbc(X)$. Here we use the convention that 
for smooth $X$ the shifted constant sheaf 
$\CC_X[ \dim X] \in \Dbc(X)$ is perverse. 
The following result is a very special case 
of \cite[Corollary 5.2.17]{Dimca}. 

\begin{lemma}\label{Lemma-1} 
Let $X$ be an algebraic variety $X$ over 
$\CC$ and $Y \subset X$ a hypersurface in it. 
Set $U=X \setminus Y$ and let $j: U 
\hookrightarrow X$ be the inclusion map. 
Then the functors 
\begin{equation}
j_!, Rj_* : \Dbc(U) \longrightarrow \Dbc(X) 
\end{equation}
preserve the perversity. 
\end{lemma}

Now for $\F \in \Dbc(X)$ 
let $\SS : X= \sqcup_{\alpha \in A} 
X_{\alpha}$ be a Whitney stratification of $X$ 
adapted to it. Then for a non-constant 
regular function 
$f : X \longrightarrow \CC$ on $X$ we define 
a subset ${\rm Sing}_{\SS}(f) \subset X$ of $X$ by 
\begin{equation}
{\rm Sing}_{\SS}(f)= \bigsqcup_{\alpha \in A} 
{\rm Sing}(f|_{X_{\alpha}}) \subset X. 
\end{equation}
We call it the stratified singular locus of 
$f$ with respect to $\SS$ (see 
\cite[Definition 4.2.7]{Dimca}). 
By the Whitney condition on $\SS$ it is 
a closed algebraic subset of $X$. By 
\cite[Propositon 4.2.8]{Dimca} we have 
\begin{equation}\label{qoo} 
{\rm supp} \varphi_f( \F ) \subset 
X_0 \cap {\rm Sing}_{\SS}(f). 
\end{equation}
Recall also that the shifted vanishing cycle functor 
\begin{equation}
^p \varphi_f ( \cdot ):= \varphi_f ( \cdot )[-1] 
: \Dbc(X) \longrightarrow \Dbc(X_0) 
\end{equation} 
preserves the perversity. Then we 
obtain the following result (see 
the proofs of \cite[Propositions 6.1.1 and 
6.1.2]{Dimca}). 

\begin{lemma}\label{Lemma-2} 
Assume that $\F$ is perverse and the dimension of 
$X_0 \cap {\rm Sing}_{\SS}(f)$ is zero. 
Then we have the concentration 
\begin{equation}
H^l  \{  ^p \varphi_f ( \F ) \}  \simeq 0 
\qquad (l \not= 0). 
\end{equation} 
\end{lemma}

\begin{proof} 
By our assumption the perverse sheaf 
$^p \varphi_f ( \F ) \in {\rm Perv}(X_0)$ is supported 
on some points in $X_0$. Then the desired 
concentration follows immediately from 
the perversity of $^p \varphi_f ( \F )$ 
(see \cite[Proposition 8.1.22]{H-T-T}). 
\end{proof}

The following lemma will be used in the 
proofs of our main theorems. 
Let $\tau$ be a strictly convex 
rational polyhedral cone in $\RR^n$ 
and $\Sigma_{\tau}$ the fan in $\RR^n$ 
formed by all its faces. Denote by 
$X_{\Sigma_{\tau}}$ the 
($n$-dimensional) toric variety 
associated to $\Sigma_{\tau}$ 
(see \cite{Fulton} and \cite{Oda} etc.). 

\begin{lemma}\label{PL} 
In the above situation, assume also that 
$\tau$ is simplicial or satisfies 
the condition 
$\dim \tau \leq 3$. Then 
the constant sheaf 
$\CC_{X_{\Sigma_{\tau}}}$ on 
$X_{\Sigma_{\tau}}$ 
is perverse (up to some shift). 
\end{lemma}  

\begin{proof} 
If $\tau$ is simplicial, then $X_{\Sigma_{\tau}}$ 
is an orbifold (see \cite[page 34]{Fulton}) 
and the assertion follows from 
\cite[Proposition 8.2.21]{H-T-T}. 
It is the case when $\dim \tau \leq 2$. 
Assume that $\dim \tau =3$. Let $T_{\tau} 
\simeq ( \CC^*)^{n- \dim \tau} 
\subset X_{\Sigma_{\tau}}$ be the (minimal) 
$T$-orbit in $X_{\Sigma_{\tau}}$ associated 
to $\tau \in \Sigma_{\tau}$ and 
$i_{\tau}: T_{\tau} \hookrightarrow  
 X_{\Sigma_{\tau}}$,
$j_{\tau}: X_{\Sigma_{\tau}} \setminus  
T_{\tau} \hookrightarrow  
 X_{\Sigma_{\tau}}$ the inclusion maps. Then 
by Fiesler \cite[Theorems 1.1 and 1.2]{F} 
we obtain
\begin{equation}
H^l i_{\tau}^{-1} R (j_{\tau})_* 
\CC_{X_{\Sigma_{\tau}} \setminus  
T_{\tau}}
\simeq \begin{cases}
\CC_{T_{\tau}} 
 & ( l=0 ), 
\\
\ 0 & ( l=1 ).  
\end{cases}
\end{equation}
This implies that we have 
\begin{equation}
H^l i_{\tau}^{!} 
\CC_{ X_{\Sigma_{\tau}} }
\simeq 0 \ \ \qquad (l < 3= 
{\rm codim} T_{\tau}). 
\end{equation}
Then the assertion follows from 
\cite[Proposition 8.1.22]{H-T-T}. 
\end{proof}

\section{Some compactifications of $\CC^n$}\label{sec:3}

In this section, we recall the constructions of 
some smooth compactifications of $\CC^n$ 
in Zaharia \cite{Zaharia} and Takeuchi-Tib{\u a}r \cite{T-T}. 
Let $f(x)=\sum_{v \in \ZZ^n_+} a_vx^v$ be a 
polynomial on $\CC^n$ ($a_v\in \CC$). 

\begin{definition}\label{dfn:3-1} 
\begin{enumerate}
\item We call the convex hull of 
$\supp(f):=\{v\in \ZZ_+^n \ |\ a_v\neq 0\} \subset 
\ZZ_+^n 
\subset \RR^n_+$ in $\RR^n$ the Newton polytope 
of $f$ and denote it by $NP(f)$.
\item (see \cite{L-S} etc.) 
We call the convex hull of $\{0\}
\cup NP(f)$ in $\RR^n$ the Newton 
polyhedron at infinity of $f$ and 
denote it by $\Gamma_{\infty}(f)$. 
\end{enumerate}
\end{definition}

For an element $u \in \RR^n$ of 
(the dual vector space of) $\RR^n$ define the 
supporting face $\gamma_u \prec  \Gamma_{\infty}(f)$ 
of $u$ in $ \Gamma_{\infty}(f)$ by 
\begin{equation}
\gamma_u = \left\{ v \in \Gamma_{\infty}(f) \ | \ 
\langle u , v \rangle 
= 
\min_{w \in \Gamma_{\infty}(f)} 
\langle u ,w \rangle \right\}. 
\end{equation}
Then we introduce an equivalence relation $\sim$ on 
(the dual vector space of) $\RR^n$ by 
$u \sim u^{\prime} \Longleftrightarrow 
\gamma_u = \gamma_{u^{\prime}}$. We can easily 
see that for any face $\gamma \prec  \Gamma_{\infty}(f)$ 
of $\Gamma_{\infty}(f)$ the closure of the 
equivalence class associated to $\gamma$ in $\RR^n$ 
is an $(n- \dim \gamma )$-dimensional rational 
convex polyhedral cone $\sigma (\gamma )$ in $\RR^n$. Moreover 
the family $\{ \sigma (\gamma ) \ | \ 
\gamma \prec  \Gamma_{\infty}(f) \}$ of cones in $\RR^n$ 
thus obtained is a subdivision of $\RR^n$. 
We call it the dual subdivision of $\RR^n$ by 
$\Gamma_{\infty}(f)$. 
If $\dim \Gamma_{\infty}(f)=n$ it 
satisfies the axiom of fans (see \cite{Fulton} and 
\cite{Oda} etc.). We call it the dual fan of 
$\Gamma_{\infty}(f)$. 

We have the following two classical 
definitions due to Kouchnirenko:

\begin{definition}[\cite{Kushnirenko}]\label{dfn:tame}
Let $\partial f\colon \CC^n 
\longrightarrow \CC^n$ be the map defined by 
$\partial f(x)=(\partial_1f(x), 
\ldots, \partial_n f(x))$. Then we say that 
$f$ is \emph{tame at infinity} if the 
restriction $(\partial f)^{-1}(B(0;\e )) 
\longrightarrow B(0;\e )$ of $\partial f$ 
to a sufficiently small ball 
$B(0;\e )$ centered at the origin 
$0 \in \CC^n$ is proper.
\end{definition}

\begin{definition}[\cite{Kushnirenko}]\label{dfn:3-3}
We say that the polynomial $f(x)=\sum_{v\in \ZZ_+^n} 
a_vx^v$ ($a_v\in \CC$) is 
\emph{non-degenerate at infinity} if for any 
face $\gamma$ of $\Gamma_{\infty}(f)$ 
such that $0 \notin \gamma$ the complex 
hypersurface $\{x \in (\CC^*)^n\ |\ 
f_{\gamma}(x)=0\}$ in $(\CC^*)^n$ is 
smooth and reduced, where we defined 
the $\gamma$-part $f_{\gamma}$ of $f$ by 
$f_{\gamma}(x)=\sum_{v \in \gamma 
\cap \ZZ_+^n} a_vx^v$.
\end{definition}

Broughton showed in \cite{Broughton} 
that if $f$ is non-degenerate 
at infinity and convenient then it is tame at infinity. 
This implies that the reduced homology of 
the general fiber of $f$ is concentrated 
in dimension $n-1$. The concentration result 
was later extended to polynomial functions 
with isolated singularities with respect 
to some fiber-compactifying extension of $f$ by 
Siersma and Tib\u ar \cite{S-T-1} and by 
Tib\u ar \cite[Theorem 4.6, Corollary 4.7]{Tibar}. 
In this paper we mainly consider 
non-convenient polynomials. 

\begin{definition}(\cite{T-T}) 
We say that a face $\gamma \prec \Gamma_{\infty}(f)$ 
is \emph{atypical} if 
$0 \in \gamma$, $\dim \gamma \geq 1$ and 
the cone $\sigma ( \gamma ) \subset \RR^n$ 
which corresponds it in the dual subdivision 
of $\Gamma_{\infty}(f)$ is not contained in the 
first quadrant $\RR^n_+$ of $\RR^n$. 
\end{definition}

This definition is related to that of 
the bad faces of $NP(f-f(0))$ 
in N{\'e}methi-Zaharia \cite{N-Z} as follows. 
If $\Delta \prec NP(f-f(0))$ 
is a bad face of $NP(f-f(0))$, then 
the convex hull $\gamma$ 
of $\{ 0 \} \cup \Delta$ in $\RR^n$ is an atypical 
one of $\Gamma_{\infty}(f)$. 
Conversely, if $\gamma \prec \Gamma_{\infty}(f)$ is 
an atypical face and $\Delta = \gamma \cap NP(f-f(0))
\prec NP(f-f(0))$ satisfies the condition 
$\dim \Delta = \dim \gamma$ then $\Delta$ is a 
bad face of $NP(f-f(0))$. 

\begin{example}\label{EXP} 
Let $n=3$ and consider a non-convenient polynomial 
$f(x,y,z)$ on $\CC^3$ whose Newton polyhedron at 
infinity $\Gamma_{\infty}(f)$ is the convex hull of 
the points $(2,0,0), (2,2,0), (2,2,3) \in \RR^3_+$ 
and the origin $0=(0,0,0) \in \RR^3$. Then the line 
segment connecting the point $(2,2,0)$ 
(resp. $(2,0,0)$) 
and the origin $0 \in \RR^3$ is an atypical 
face of $\Gamma_{\infty}(f)$.  However the 
triangle whose vertices are the points 
$(2,0,0)$, $(2,2,0)$ 
and the origin $0 \in \RR^3$ is not so. 
Note that for the line segment $\gamma$ 
connecting $(2,0,0)$ and the origin 
we have $\dim \sigma ( \gamma ) 
\cap \RR_+^3 =2$. 
\end{example}

From now we recall the smooth compactifications of 
$\CC^n$ in \cite{Zaharia} and \cite{T-T} 
(for their applications to monodromies at 
infinity see \cite{E-T}, \cite{M-T-2}, 
\cite{M-T-4} and \cite{T-T}). 
Assume that the polynomial 
$f(x)=\sum_{v \in \ZZ^n_+} a_vx^v  
\in \CC[x_1,\ldots,x_n]$ is 
``non-convenient" and 
$\dim \Gamma_{\infty}(f)=n$. 
Let $\Sigma_0$ be the dual 
fan of $\Gamma_{\infty}(f)$. 
Assume also that $f$ is non-degenerate at infinity. 
We consider $\CC^n$ as a 
toric variety associated with the fan 
$\Xi$ in $\RR^n$ formed by all the 
faces of the first quadrant 
$\RR_+^n \subset \RR^n$. 
Denote by $T \simeq (\CC^*)^n$ the 
open dense torus in it. Let $\Sigma_1$ 
be a subdivision of the dual fan $\Sigma_0$ 
of $\Gamma_{\infty}(f)$ which contains 
$\Xi$ as its subfan. Then we can 
construct a smooth subdivision $\Sigma$ of 
$\Sigma_1$ without subdividing 
the cones in $\Xi$ (see e.g. 
\cite[Lemma (2.6), Chapter II, page 
99]{Oka}). 
This implies that the 
toric variety $X_{\Sigma}$ associated with 
$\Sigma$ is a smooth compactification 
of $\CC^n$. This construction of $X_{\Sigma}$ 
coincides with the one in Zaharia \cite{Zaharia}. 
Recall that $T$ acts on 
$X_{\Sigma}$ and the $T$-orbits are 
parametrized by the cones in $\Sigma$. 
For a cone $\sigma \in \Sigma$ denote 
by $T_{\sigma} \simeq (\CC^*)^{n-\dim 
\sigma}$ the corresponding $T$-orbit. 
If $\sigma^{\perp} \simeq \RR^{n-\dim 
\sigma}$ is the orthogonal 
complement of (the affine span of) $\sigma$ 
we have $T_{\sigma}= 
\Spec ( \CC [\sigma^{\perp} \cap \ZZ^n])$. 
There exist also natural affine open 
subsets $\CC^n(\sigma) \simeq \CC^n$ of 
$X_{\Sigma}$ associated to 
$n$-dimensional cones $\sigma$ in $\Sigma$ 
as follows. Let $\sigma$ be an 
$n$-dimensional (smooth) cone in $\Sigma$ and 
$\{w_1,\ldots, w_n\} \subset \ZZ^n$ the 
set of the (non-zero) 
primitive vectors on the edges of 
$\sigma$. Let $\sigma^{\circ}$ be the dual 
cone of $\sigma$. 
Then by the smoothness of $\sigma$ 
the semigroup ring $\CC [\sigma^{\circ} \cap \ZZ^n ]$ 
is isomorphic to the polynomial ring 
$\CC [y_1, \ldots, y_n]$. 
This implies that the affine open 
subset $\CC^n(\sigma):= 
\Spec ( \CC [ \sigma^{\circ} \cap \ZZ^n  ])$ 
of $X_{\Sigma}$ is isomorphic to $\CC^n_y$. 
Moreover, on $\CC^n(\sigma) \simeq 
\CC^n_y$ the function $f(x)=
\sum_{v \in \ZZ^n_+} a_vx^v$ has the 
following form:
\begin{equation}\label{FOM} 
f(y)=\sum_{v \in \ZZ_+^n}
a_{v}y_1^{\langle w_1,v \rangle}\cdots y_n^{\langle 
w_n, v \rangle} =y_1^{b_1} \cdots y_n^{b_n} 
\times f_{\sigma}(y),
\end{equation}
where we set 
\begin{equation}
b_i=\min_{v\in \Gamma_{\infty}(f)} 
\langle w_i,v \rangle \leq 0 \qquad 
(i=1,2,\ldots,n)
\end{equation}
and $f_{\sigma}(y)$ is a polynomial on 
$\CC^n(\sigma) \simeq \CC^n_y$. In 
$\CC^n(\sigma) \simeq \CC^n_y$ 
the hypersurface $Z:= \overline{f^{-1}(0)} 
\subset X_{\Sigma}$ is explicitly 
written as $\{ y \in \CC^n(\sigma) \ | \ 
f_{\sigma}(y)=0 \}$. By \eqref{FOM} we see 
that $f$ is extended to a meromorphic 
function on $\CC^n(\sigma) \simeq \CC^n_y$. 
The variety $X_{\Sigma}$ is covered by 
such affine open 
subsets. Let $\tau$ be a 
$d$-dimensional face of the $n$-dimensional cone 
$\sigma \in \Sigma$. For simplicity, 
assume that $w_1,\ldots, w_d$ generate 
$\tau$. Then in the affine chart 
$\CC^n(\sigma) \simeq \CC^n_y$ the 
$T$-orbit $T_{\tau}$ associated to 
$\tau$ is explicitly defined by
\begin{equation*}
T_{\tau}=\{(y_1,\ldots,y_n)\in 
\CC^n(\sigma)\  |\ y_1=\cdots =y_d=0,\ 
y_{d+1},\ldots, y_n\neq 0\}\simeq (\CC^*)^{n-d}.
\end{equation*}
Hence we have
\begin{equation}
 X_{\Sigma}=\bigcup_{\dim \sigma =n}
\CC^n(\sigma)=\bigsqcup_{\tau 
\in\Sigma}T_{\tau}.
\end{equation}
Now $f$ extends to a meromorphic 
function on $X_{\Sigma}$, 
which may still have points of indeterminacy. 
For simplicity we denote this 
meromorphic extension also by $f$. 
From now on, we will eliminate 
its points of indeterminacy 
by blowing up $X_{\Sigma}$ 
(see \cite[Section 3]{M-T-2} 
and \cite[Section 3]{M-T-4} for 
the details). 
For a cone $\sigma$ in $\Sigma$ by 
taking a non-zero vector $u$ in the 
relative interior $\relint(\sigma)$ of 
$\sigma$ we define a face $\gamma_{\sigma}$ 
of $\Gamma_{\infty}(f)$ by 
$\gamma_{\sigma}= \gamma_u$. 
Note that $\gamma_{\sigma}$ does 
not depend on the choice of $u \in 
\relint(\sigma)$. We call it 
the supporting face of $\sigma$ in 
$\Gamma_{\infty}(f)$. Following 
Libgober-Sperber \cite{L-S}, 
we say that a $T$-orbit 
$T_{\sigma}$ in $X_{\Sigma}$ 
(or a cone $\sigma \in \Sigma$) is at infinity 
if the supporting face $\gamma_{\sigma} \prec 
\Gamma_{\infty}(f)$ satisfies the condition 
$0 \notin \gamma_{\sigma}$. 
We can easily see that $f$ has poles on 
the union of $T$-orbits at infinity as follows. 
Let $\rho_1, \rho_2, \ldots, \rho_r$ be 
the $1$-dimensional cones 
at infinity in $\Sigma$. 
Then $T_{\rho_1},T_{\rho_2},\ldots,T_{\rho_r}$ are the 
$(n-1)$-dimensional $T$-orbits at infinity 
in $X_{\Sigma}$. For any 
$i=1,2,\ldots,r$ the toric divisor 
$D_i:=\overline{T_{\rho_i}}$ is a smooth 
hypersurface in $X_{\Sigma}$. 
Let us denote the (unique non-zero) 
primitive vector in $\rho_i \cap 
\ZZ^n$ by $u_i$. Then the order $a_i>0$ 
of the pole of $f$ along $D_i$ 
is given by
\begin{equation}
a_i=-\min_{v\in \Gamma_{\infty}(f)} 
\langle u_i,v \rangle.
\end{equation}
From this we see that the 
poles of $f$ are contained in 
the normal crossing divisor $D:= D_1 
\cup \cdots \cup D_r$. 
Moreover by the non-convenience of $f$, there 
exist some cones $\sigma \in \Sigma$ such that 
$\sigma \notin \Xi$ and $0 \in \gamma_{\sigma}$ 
i.e. $\gamma_{\sigma}$ is an atypical face of 
$\Gamma_{\infty}(f)$. For such $\sigma$ the 
function $f$ extends holomorphically to a 
neighborhood of $T_{\sigma} \subset 
X_{\Sigma} \setminus \CC^n$. For this 
reason we call them \emph{horizontal 
$T$-orbits} in $X_{\Sigma}$ (in the tame 
case where $f$ is convenient, they do not 
appear). 
Note also that by the 
non-degeneracy at infinity of $f$, for 
any non-empty subset $I \subset \{ 
1,2, \ldots, r \}$ 
the hypersurface $Z= \overline{f^{-1}(0)}$ 
in $X_{\Sigma}$ intersects $D_I:= 
\bigcap_{i \in I}D_i$ transversally 
(or the intersection is empty). At such intersection 
points, $f$ has indeterminacy. 
We can easily see that the 
meromorphic extension of $f$ 
to $X_{\Sigma}$ has points of 
indeterminacy in the subvariety 
$D \cap Z$ of $X_{\Sigma}$ of 
codimension two. 
Now, in order to eliminate the 
indeterminacy of the meromorphic function 
$f$ on $X_{\Sigma}$, we first 
consider the blow-up $\pi_1 \colon 
X_{\Sigma}^{(1)} \longrightarrow X_{\Sigma}$ 
of $X_{\Sigma}$ along the 
$(n-2)$-dimensional smooth subvariety 
$D_1\cap Z$. Then the indeterminacy of 
the pull-back $f \circ \pi_1$ of 
$f$ to $X_{\Sigma}^{(1)}$ is 
improved. If $f \circ \pi_1$ still 
has points of indeterminacy on the 
intersection of the exceptional divisor 
$E_1$ of $\pi_1$ and the proper 
transform $Z^{(1)}$ of $Z$, we construct 
the blow-up $\pi_2 \colon 
X_{\Sigma}^{(2)} \longrightarrow 
X_{\Sigma}^{(1)}$ of $X_{\Sigma}^{(1)}$ 
along $E_1 \cap Z^{(1)}$. By repeating 
this procedure $a_1$ times, we obtain 
a tower of blow-ups
\begin{equation}
X_{\Sigma}^{(a_1)} 
\underset{\pi_{a_1}}{\longrightarrow}
\cdots \cdots
\underset{\pi_2}{\longrightarrow} X_{\Sigma}^{(1)}
\underset{\pi_1}{\longrightarrow} X_{\Sigma}.
\end{equation}
For the details 
see the figures in \cite[page 420]{M-T-2}. 
Then the pull-back of $f$ to $X_{\Sigma}^{(a_1)}$ 
has no indeterminacy 
over $T_{\rho_1}$. It also extends to a holomorphic function 
on (an open dense subset of) 
the exceptional divisor of the last blow-up 
$\pi_{a_1}$. 
For this reason we call it and its 
proper transform $F_1$ in the variety 
$\tl{X_{\Sigma}}$ that we construct below 
\emph{horizontal exceptional divisors}. 
Note that for any $t \in \CC$ 
the closure of the hypersurface 
$f^{-1}(t) \subset \CC^n$ in $X_{\Sigma}^{(a_1)}$ 
intersects $F_1$ 
transversally. Moreover it does not intersect 
the other exceptional divisors.

Next we apply this construction 
to the proper transforms of $D_2$ and $Z$ in 
$X_{\Sigma}^{(a_1)}$. Then we obtain 
also a tower of blow-ups
\begin{equation}
X_{\Sigma}^{(a_1)(a_2)} \longrightarrow 
\cdots \cdots \longrightarrow 
X_{\Sigma}^{(a_1)(1)} \longrightarrow 
X_{\Sigma}^{(a_1)}
\end{equation}
and the indeterminacy of the pull-back 
of $f$ to 
$X_{\Sigma}^{(a_1)(a_2)}$ is eliminated 
over $T_{\rho_1} \sqcup T_{\rho_2}$. By applying 
the same construction to (the proper 
transforms of) $D_3, D_4,\ldots, D_r$, 
we finally obtain a proper morphism 
$\pi \colon \tl{X_{\Sigma}} 
\longrightarrow X_{\Sigma}$ such that 
$g:=f \circ \pi$ has no point of 
indeterminacy on the whole $\tl{X_{\Sigma}}$. 
Note that the smooth 
compactification $\tl{X_{\Sigma}}$ of $\CC^n$ 
thus obtained is not a toric 
variety any more. By 
constructing a blow-up $\tl{X_{\Sigma}} 
\longrightarrow X_{\Sigma}$ of $X_{\Sigma}$ 
to eliminate the indeterminacy of $f$ we 
thus obtain a commutative diagram: 
\begin{equation}
\begin{CD}
\CC^n  @>{\iota}>> \tl{X_{\Sigma}} 
\\
@V{f}VV   @VV{g}V
\\
\CC  @>>{j}>  \PP^1 
\end{CD}
\end{equation}
of holomorphic maps, 
where $\iota : \CC^n  
\hookrightarrow \tl{X_{\Sigma}}$ and 
$j : \CC  
\hookrightarrow \PP^1$ are 
the inclusion maps and 
$g$ is proper. 
On $\tl{X_{\Sigma}}$ 
we have constructed also $r$ 
(smooth) horizontal exceptional divisors 
$F_1, F_2, \ldots, F_r$. 
The other exceptional divisors in $\tl{X_{\Sigma}}$ 
are called 
\emph{intermediate exceptional divisors}. 
By our construction 
of the blow-up $\pi \colon \tl{X_{\Sigma}} 
\longrightarrow X_{\Sigma}$, 
$F_1 \cup F_2 \cup \cdots \cup F_r$ 
is a normal crossing divisor 
in $\tl{X_{\Sigma}}$ and 
for any non-empty subset 
$I \subset \{ 1,2, \ldots, r \}$ 
and $t \in \CC$ the hypersurface 
$g^{-1}(t) \subset \tl{X_{\Sigma}}$ 
intersects $F_I:= \cap_{i \in I} F_i$ 
transversally. Moreover 
$g^{-1}(t)$ does not intersect 
intermediate exceptional divisors. 
For a point $b \in \CC$ define a 
function $h: \CC \longrightarrow 
\CC$ on $\CC$ by $h(t)=t-b$ so that 
we have $h^{-1}(0)= \{ b \}$. Then by the 
above-mentioned 
property of $F_i$ and \eqref{qoo} 
the support of 
the constructible sheaf 
$\varphi_{h \circ g}( \iota_! 
\CC_{\CC^n})$ does not intersect 
the union of the exceptional divisors 
in $\pi \colon \tl{X_{\Sigma}} 
\longrightarrow X_{\Sigma}$. 
Moreover, for the pole divisor $D
= D_1 \cup \cdots \cup D_r 
\subset X_{\Sigma}$ of (the 
meromorphic extension of) $f$ to 
$X_{\Sigma}$, the support does not intersect 
$\pi^{-1}(D)$.

\section{Bifurcation sets of 
polynomial functions}\label{sec:4}

In this section we study the bifurcation 
values of polynomial functions. 
Let $f: \CC^n \longrightarrow \CC$ be a 
polynomial function. Throughout this section 
we assume that $f$ is non-degenerate 
at infinity and $\dim \Gamma_{\infty}(f)=n$. 
Let $\Sigma_0$ be the dual 
fan of $\Gamma_{\infty}(f)$. 
Let $\gamma_1, \ldots, \gamma_m$ be the atypical faces of 
$\Gamma_{\infty}(f)$. 
For $1 \leq i \leq m$ let $K_i \subset \CC$ 
be the set of the critical values of the $\gamma_i$-part 
\begin{equation}
f_{\gamma_i}: T= (\CC^*)^n \longrightarrow \CC
\end{equation}
of $f$. We denote by 
${\rm Sing} f \subset \CC^n$ the set of the 
critical points of 
$f \colon \CC^n \longrightarrow \CC$ and set 
\begin{equation}
K_f = f( {\rm Sing} f) \cup \{ f(0) \} \cup 
( \cup_{i=1}^m K_i). 
\end{equation}
Then the following result was obtained 
by N{\'e}methi-Zaharia \cite{N-Z}. 

\begin{theorem}\label{BES}
(N{\'e}methi-Zaharia \cite{N-Z}) 
In the situation above, we have 
$B_f \subset K_f$. 
\end{theorem}

\begin{remark}
If for an atypical face $\gamma_i$ 
of $\Gamma_{\infty}(f)$ the face 
$\Delta = \gamma_i \cap NP(f-f(0)) 
\prec NP(f-f(0))$ of $NP(f-f(0))$ 
is not bad in the sense 
of N{\'e}methi-Zaharia \cite{N-Z}, 
then $\dim NP(f_{\gamma_i} -f(0))= \dim \Delta 
< \dim \gamma_i$, $f_{\gamma_i} - f(0)$ is 
a positively homogeneous Laurent polynomial 
on $T= (\CC^*)^n$ and 
we have $K_i = \{ f(0) \}$. 
Therefore the above inclusion $B_f \subset K_f$ 
coincides with the one in \cite{N-Z}. 
\end{remark} 
Moreover the authors of \cite{N-Z} 
proved the equality $B_f=K_f$ for $n=2$ 
and conjectured its validity 
in higher dimensions. Later Zaharia \cite{Zaharia} 
proved it for any $n \geq 2$ but 
under some supplementary assumptions on $f$. 
By using the definitions and the notations 
in Section \ref{sec:1} 
we can improve his result as follows. 

\begin{theorem}\label{SMT-1} 
Assume that $f$ has isolated singularities 
at infinity over $b \in K_f \setminus 
[f( {\rm Sing} f) \cup \{ f(0) \} ]$ 
and for any $1 \leq i \leq m$ such that $b \in K_i$ 
the relative interior $\relint ( \gamma_i)$ 
of $\gamma_i \prec \Gamma_{\infty}(f)$ 
is contained in $\Int ( \RR_+^n)$. 
Assume also that there exists $1 \leq i \leq m$ 
such that $b \in K_i$ and 
$\gamma_i \prec \Gamma_{\infty}(f)$ 
is relatively simple. 
Then we have $E_f(b)>0$ and hence $b \in B_f$. 
\end{theorem}

\begin{proof} 
By our assumption, for any $1 \leq i \leq m$ 
the hypersurface $f_{\gamma_i}^{-1} 
(b) \subset T_{i} 
\simeq (\CC^*)^{\dim \gamma_i}$ in $T_{i}
= \Spec ( \CC [L_{\gamma_i} \cap \ZZ^n]) $ 
has only isolated singular points 
at $p_{i,1}, \ldots, p_{i, n_i}$. Here some 
$n_i$ can be zero. Obviously we have 
$n_i>0$ if and only if $b \in K_i$.  
From now we shall freely use the 
smooth compactification 
$\tl{X_{\Sigma}}$ of $\CC^n$ 
and the notations related to it 
in Section \ref{sec:3}. 
Let ${\rm Cone}_{\infty}(f) \subset \RR^n_v$ 
be the cone generated by $\Gamma_{\infty}(f)$. 
We define its dual cone $C \subset \RR^n_u$ by 
\begin{equation}
C= \{ u \in \RR^n \ | \ \langle u, v \rangle 
 \geq 0 \ \ \text{for any} \ v \in 
 {\rm Cone}_{\infty}(f) \}. 
\end{equation}
Then a cone $\sigma \in \Sigma$ is 
at infinity if and only if 
it is not contained in $C$.  
We shall prove that 
the jump $E_f(b) \in \ZZ$ 
of the constructible function on $\CC$ 
\begin{equation}
\chi_c(t)= \sum_{j \in \ZZ} (-1)^j \dim 
H^j_c ( f^{-1}(t); \CC ) \ \ \ \ (t \in \CC )
\end{equation} 
at the point $b \in K_f \setminus 
[f( {\rm Sing} f) \cup \{ f(0) \} ]$ 
is positive. 
For the point $b \in \CC$ define a 
function $h: \CC \longrightarrow 
\CC$ on $\CC$ by $h(t)=t-b$ so that 
we have $h^{-1}(0)= \{ b \}$. Then 
we have 
\begin{equation}
E_f(b)= (-1)^{n-1} 
\sum_{j \in \ZZ} (-1)^j \dim 
H^j \varphi_h(R f_! \CC_{\CC^n})_b,  
\end{equation}
where $\varphi_h: \Dbc (\CC ) \longrightarrow 
\Dbc ( \{ b \} )$ is Deligne's vanishing cycle 
functor associated to $h$. Since we have 
$f = g \circ \iota$ on a neighborhood of 
$b \in K_f \setminus 
[f( {\rm Sing} f) \cup \{ f(0) \} ]$ 
and $g$ is proper, by Proposition \ref{PDI} 
we obtain an isomorphism 
\begin{equation}
\varphi_h(R  f_! \CC_{\CC^n}) \simeq 
R(g|_{g^{-1}(b)})_* \varphi_{h \circ g}
( \iota_! \CC_{\CC^n}). 
\end{equation}
This implies that for the constructible function 
$\chi \{ \varphi_{h \circ g}
( \iota_! \CC_{\CC^n}) \} 
\in F_{\ZZ} (g^{-1}(b))$ 
on $g^{-1}(b)=(h \circ g)^{-1}(0) 
\subset \tl{X_{\Sigma}}$ we have 
\begin{equation}
\sum_{j \in \ZZ} (-1)^j \dim 
H^j \varphi_h(R f_! \CC_{\CC^n})_b 
 = 
\int_{g^{-1}(b)} \chi \{ \varphi_{h \circ g}
( \iota_! \CC_{\CC^n}) \}. 
\end{equation}
Hence for the calculation of $E_f(b)$, 
it suffices to calculate 
\begin{equation}
\chi \{ \varphi_{h \circ g}
( \iota_! \CC_{\CC^n}) \} (p) = 
\sum_{j \in \ZZ} (-1)^j \dim 
 H^j \varphi_{h \circ g}
( \iota_! \CC_{\CC^n})_p  
\end{equation}
at each point $p$ of $g^{-1}(b)$. 
Let $\Sigma_C$ (resp. $\Sigma_C^{\prime}$) 
be the fan formed by all the faces of 
the cone $C$ (resp. by all the cones 
in $\Sigma$ contained in $C$)  
and denote by $X_{\Sigma_C}$ 
(resp. $X_{\Sigma_C^{\prime}}$)  
the possibly singular (resp. smooth) 
toric variety associated to it. 
Then $X:= X_{\Sigma_C^{\prime}} =  
\sqcup_{\sigma \subset C} T_{\sigma}$ 
is an open subset of $X_{\Sigma}$ 
and there exists a natural proper 
morphism 
\begin{equation}
\pi : 
X= X_{\Sigma_C^{\prime}} \longrightarrow 
X_{\Sigma_C}
\end{equation}
of toric varieties. Note that 
for the pole divisor $D$ of (the 
meromorphic extension of) $f$ to 
$X_{\Sigma}$ (see Section \ref{sec:3}) we have $X= 
X_{\Sigma} \setminus D$. Recall also that 
the centers of the blow-ups in the 
construction of $\tl{X_{\Sigma}} 
\longrightarrow X_{\Sigma}$ are 
above $D=X_{\Sigma} \setminus X$. 
Hence we can consider $X$ also as an open subset 
of $\tl{X_{\Sigma}}$. 
Since the Newton polytope 
$NP(f)$ of $f$ is contained in the 
dual cone $C^{\circ} = {\rm Cone}_{\infty}(f)$ 
of $C$ and 
\begin{equation}
X_{\Sigma_C} =  
\Spec ( \CC [ C^{\circ} \cap \ZZ^n]),  
\end{equation}
we can naturally regard $f$ as regular 
functions on $X_{\Sigma_C}$ and 
$X= X_{\Sigma_C^{\prime}}$. This implies 
that $X= X_{\Sigma_C^{\prime}}$ is an 
open subset of  $g^{-1}( \CC ) \cap 
\tl{X_{\Sigma}}$. In particular, 
if $\sigma \in \Sigma_C^{\prime}$ is not 
contained in $\RR^n_+$ 
then $T_{\sigma} \subset X \setminus \CC^n$ 
and $f$ extends holomorphically to 
$T_{\sigma}$. Namely $T_{\sigma}$ 
is a horizontal 
$T$-orbit in $X \setminus \CC^n$. 
By our assumption $b \notin f( {\rm Sing} f)$ 
and the result 
at the end of Section \ref{sec:3}, we 
see also that the support of the 
constructible sheaf 
$\varphi_{h \circ g}( \iota_! \CC_{\CC^n}) 
\in \Dbc (g^{-1}(b))$ is contained in 
$(X \setminus \CC^n) \cap g^{-1}(b)$. 
We thus obtain an equality 
\begin{equation}
E_f(b) = (-1)^{n-1} 
\int_{(X \setminus \CC^n) \cap g^{-1}(b)} 
\chi \{ \varphi_{h \circ g}
( \iota_! \CC_{\CC^n}) \}. 
\end{equation}
Namely, for the calculation of $E_f(b)$ 
it suffices to calculate the constructible 
function $\chi \{ \varphi_{h \circ g}
( \iota_! \CC_{\CC^n}) \}$ 
only on $T$-orbits 
in $X \setminus \CC^n$ associated to 
the cones $\sigma \in 
\Sigma_C^{\prime} \subset \Sigma$ such that 
$\relint ( \sigma ) \subset C  
\setminus \RR_+^n$. For $\sigma \in 
\Sigma_C^{\prime} \subset \Sigma$ such that 
$\relint ( \sigma ) \subset \Int (C) 
\setminus \RR_+^n$ we have $\gamma_{\sigma} 
= \{ 0 \} \prec \Gamma_{\infty}(f)$ and the 
restriction of $g|_X: X \longrightarrow \CC$ 
to the $T$-orbit $T_{\sigma} \subset X$ 
is the constant function $f(0) \in \CC$. 
Hence we get $g^{-1}(b) \cap T_{\sigma} 
= \emptyset$ for the point $b \in K_f 
\setminus [ f({\rm Sing} f) \cup \{ f(0) \} ]$. 
For $1 \leq i \leq m$ let 
$\sigma_i= \sigma ( \gamma_i) \in \Sigma_0$ 
be the cone which corresponds 
to $\gamma_i$ in the dual fan $\Sigma_0$ 
of $\Gamma_{\infty}(f)$. Recall that by 
the definition of atypical faces we have 
$0 \in \gamma_i$ and the face 
$\sigma_i \prec C$ of $C$ is not contained 
in $\RR^n_+$. For $\sigma \in 
\Sigma_C^{\prime} \subset \Sigma$ such that 
$\relint ( \sigma ) \subset \partial C  
\setminus \RR_+^n$ there exists unique 
$1 \leq i \leq m$ for which we have 
$\relint ( \sigma ) \subset 
\relint ( \sigma_i )$. If $\dim \sigma 
= \dim \sigma_i$ we have an isomorphism 
$T_{\sigma} \simeq T_{i}= 
\Spec ( \CC [L_{\gamma_i} \cap \ZZ^n]) \simeq 
( \CC^*)^{\dim \gamma_i}$ and the 
restriction of $g|_X: X \longrightarrow \CC$ 
to $T_{\sigma} \subset X$ is naturally 
identified with 
$f_{\gamma_i}: T_i \longrightarrow \CC$. 
This implies that the hypersurface 
$g^{-1}(b) \cap T_{\sigma} \subset 
T_{\sigma} \simeq T_{i}$ has only isolated 
singular points $p_{i,1}, \ldots, p_{i, n_i} 
\in T_{\sigma} \simeq T_{i}$ and 
\begin{equation}
T_{\sigma} \cap \ \supp \ 
\varphi_{h \circ g} ( \iota_! \CC_{\CC^n}) 
\subset \{  p_{i,1}, \ldots, p_{i, n_i} \} 
\end{equation}
in this case. On the other hand, if 
$\dim \sigma < \dim \sigma_i$ we have 
$\dim T_{\sigma} > \dim T_i$ and for the 
hypersurface 
$g^{-1}(b) \cap T_{\sigma} \subset 
T_{\sigma}$ there exists an isomorphism 
\begin{equation}\label{NIS} 
g^{-1}(b) \cap T_{\sigma} \simeq 
f_{\gamma_i}^{-1}(b) \times 
( \CC^*)^{\dim T_{\sigma} - \dim T_i}. 
\end{equation}
This implies that 
$g^{-1}(b) \cap T_{\sigma} \subset 
T_{\sigma}$ has non-isolated singular points  
if $n_i>0$. From now on, we shall overcome 
this difficulty by using Proposition \ref{PDI}. 
For $1 \leq i \leq m$ let $\Sigma_i$ 
be the fan in $\RR^n$ formed by 
all the faces of $\sigma_i$ and 
denote by $X_{\Sigma_i}$ 
the (possibly singular) 
toric variety associated to it. 
Then $X_{\Sigma_i}$ is an open 
subset of $X_{\Sigma_C}$. 
Let $\sigma_i^{\circ} 
\subset \RR^n$ be the dual 
cone of $\sigma_i$ in $\RR^n$. Then 
$\sigma_i^{\circ} \simeq C_i \times 
\RR^{\dim \gamma_i}$ for a proper convex cone 
$C_i$ in $\RR^{n- \dim \gamma_i}$ and 
we have an isomorphism 
\begin{equation}
X_{\Sigma_i} \simeq 
\Spec ( \CC [ \sigma_i^{\circ} \cap \ZZ^n]). 
\end{equation}
Note that the (minimal) $T$-orbit $T_{\sigma_i}$ in 
$X_{\Sigma_i}$ which corresponds to 
$\sigma_i \in \Sigma_i$ is naturally 
identified with $T_{i}= 
\Spec ( \CC [L_{\gamma_i} \cap \ZZ^n]) \simeq 
( \CC^*)^{\dim \gamma_i}$. 
More precisely $X_{\Sigma_i}$ is the product 
$X_i \times T_{\sigma_i}$ of the 
$(n - \dim \gamma_i)$-dimensional affine toric 
variety $X_i= 
\Spec ( \CC [C_i \cap \ZZ^{n- \dim \gamma_i} ])$ 
and $T_{\sigma_i} \simeq T_{i} 
\simeq (\CC^*)^{\dim \gamma_i}$. 
Since $NP(f) \subset \sigma_i^{\circ}$ and 
$f \in \CC [ \sigma_i^{\circ} \cap \ZZ^n]$, we 
can naturally regard $f$ as a regular 
function on $X_{\Sigma_i}$. We denote it by 
$f_i:X_{\Sigma_i} \longrightarrow \CC$. 
For $1 \leq i \leq m$ let 
$\Sigma_i^{\prime} \subset \Sigma$ 
be the subfan of $\Sigma$ consisting of 
the cones in $\Sigma$ contained in 
$\sigma_i$ and 
denote by $X_{\Sigma_i^{\prime}}$ 
the smooth toric variety associated to it. 
Then $X_{\Sigma_i^{\prime}}$ is an open 
subset of $X \subset \tl{X_{\Sigma}}$ and 
there exists a proper 
morphism 
\begin{equation}
\pi_i : 
X_{\Sigma_i^{\prime}} \longrightarrow 
X_{\Sigma_i}
\end{equation}
of toric varieties. 
Moreover we have a commutative diagram 
\begin{equation}
\begin{CD}
X_{\Sigma_i^{\prime}} @>>> 
X= X_{\Sigma_C^{\prime}}  
\\
@V{\pi_i}VV   @VV{\pi}V
\\
X_{\Sigma_i}  @>>> X_{\Sigma_C} 
\end{CD}
\end{equation}
such that $\pi^{-1} X_{\Sigma_i} 
=X_{\Sigma_i^{\prime}} \subset X$, 
where the horizontal arrows are 
the inclusion maps. It is also easy to 
see that the closed subset 
$(X \setminus \CC^n) \cap g^{-1}(b)$ 
of $X$ is covered by 
the affine open subvarieties 
$X_{\Sigma_1^{\prime}}, \ldots, 
X_{\Sigma_m^{\prime}} \subset X$. 
Note that for the 
restriction $g_i=g|_{X_{\Sigma_i^{\prime}}}: 
X_{\Sigma_i^{\prime}} \longrightarrow \CC$ 
of $g|_X$ we have $g_i=f_i \circ \pi_i$. 
Then by applying Proposition \ref{PDI} to the 
proper morphism $\pi_i : 
X_{\Sigma_i^{\prime}} \longrightarrow 
X_{\Sigma_i}$ we obtain an isomorphism 
\begin{equation}
R(\pi_i|_{g_i^{-1}(b)})_* \varphi_{h \circ g_i}
( \iota_! \CC_{\CC^n}|_{X_{\Sigma_i^{\prime}}} ) 
\simeq 
\varphi_{h \circ f_i} \left\{ 
R( \pi_i)_* 
( \iota_! \CC_{\CC^n}|_{X_{\Sigma_i^{\prime}}} ) 
\right\}. 
\end{equation}
The advantage to consider 
$\varphi_{h \circ f_i} \{ R( \pi_i)_* 
( \iota_! \CC_{\CC^n}|_{X_{\Sigma_i^{\prime}}} ) 
\}$ instead of $\varphi_{h \circ g_i}
( \iota_! \CC_{\CC^n}|_{X_{\Sigma_i^{\prime}}} ) $ 
is that its support is a discrete subset of 
$f_i^{-1}(b) \subset X_{\Sigma_i} 
\subset X_{\Sigma_C}$ 
by our assumption that $f$ has isolated singularities 
at infinity over $b \in K_f \setminus 
[f( {\rm Sing} f) \cup \{ f(0) \} ]$.  
Set 
\begin{equation}
\F_i = R( \pi_i)_* 
( \iota_! \CC_{\CC^n}|_{X_{\Sigma_i^{\prime}}} ) 
\simeq 
R( \pi_i )_! \CC_{\CC^n \cap 
X_{\Sigma_i^{\prime}}} 
 \in \Dbc ( X_{\Sigma_i} ). 
\end{equation}
Then the topological 
integral 
\begin{equation}
\int_{g^{-1}(b)} 
\chi \{ \varphi_{h \circ g}
( \iota_! \CC_{\CC^n}) \} = 
\int_{(X \setminus \CC^n) \cap g^{-1}(b)} 
\chi \{ \varphi_{h \circ g}
( \iota_! \CC_{\CC^n}) \} 
\end{equation}
is equal to 
\begin{equation}
\sum_{i=1}^m 
\sum_{j=1}^{n_i} \chi \{ \varphi_{h \circ f_i} 
( \F_i )_{p_{i,j}} \}. 
\end{equation}
If $b \notin K_i$ ($\Longleftrightarrow 
n_i=0$) we have $\varphi_{h \circ f_i} 
( \F_i ) \simeq 0$ on a neighborhood of 
$T_{\sigma_i} \subset X_{\Sigma_i}$. 
Let us consider the remaining case where 
$b \in K_i$ ($\Longleftrightarrow 
n_i>0$). Then by our assumption 
$\relint ( \gamma_i) \subset \Int ( \RR_+^n)$ 
we have $\sigma_i \cap \RR_+^n = \{ 0 \}$. 
This implies that for the embedding 
$\iota_i: T=( \CC^*)^n \hookrightarrow 
X_{\Sigma_i}$ there exists an isomorphism 
$\F_i \simeq ( \iota_i)_! \CC_T$. 
Hence by Lemma \ref{Lemma-1}, $\F_i$ 
is a perverse sheaf on 
$X_{\Sigma_i}$ (up to some shift). Since 
the support of $\varphi_{h \circ f_i} 
( \F_i ) $ is discrete, by 
Lemma \ref{Lemma-2} 
we thus obtain the concentration 
\begin{equation}
H^l \varphi_{h \circ f_i} 
( \F_i )_{p_{i,j}} \simeq 0 \ \ \ \ 
(l \not= n-1) 
\end{equation}
for any $1 \leq j \leq n_i$. Set 
$\mu_{i,j}= \dim H^{n-1} \varphi_{h \circ f_i} 
( \F_i )_{p_{i,j}} \geq 0$. Then 
$E_f(b)$ can be expressed as a sum of 
non-negative integers as follows:  
\begin{equation}
E_f(b)= (-1)^{n-1} 
\int_{(X \setminus \CC^n) \cap g^{-1}(b)} 
\chi \{ \varphi_{h \circ g}
( \iota_! \CC_{\CC^n}) \} = \sum_{i=1}^m 
\sum_{j=1}^{n_i} \mu_{i,j}. 
\end{equation}
By our assumption 
there exists $1 \leq i \leq m$ 
such that $n_i>0$ 
($\Longleftrightarrow b \in K_i$) and 
$\gamma_i \prec \Gamma_{\infty}(f)$ 
is relatively simple. Then the cone 
$\sigma_i \in \Sigma_0$ 
satisfies the condition 
$\sigma_i \cap \RR_+^n = \{ 0 \}$. 
For a face $\tau \prec \sigma_i$ of $\sigma_i$ 
we set $Y_{\tau}=  
\overline{T_{\tau}} \subset X_{\Sigma_i}$ 
and $f_{\tau}=f_i|_{Y_{\tau}}: Y_{\tau} \longrightarrow \CC$. 
Note that we have 
$T_{\sigma_i} = Y_{ \sigma_i }$. 
Then for any $1 \leq j \leq n_i$ we can 
easily show that 
$(-1)^{n-1} \mu_{i,j} = 
\chi \{ \varphi_{h \circ f_i} 
( \F_i )_{p_{i,j}} \} = 
\chi \{ \varphi_{h \circ f_i} 
( ( \iota_i)_! \CC_T )_{p_{i,j}} \}$ 
is equal to the alternating sum 
\begin{equation}
\sum_{\tau \prec \sigma_i}  (-1)^{\dim \tau}  
\chi \{ \varphi_{h \circ f_{\tau}} 
( \CC_{Y_{\tau}} )_{p_{i,j}} \}. 
\end{equation}
Here we used the additivity of the vanishing 
cycle functor $\varphi_{h \circ f_i} ( \cdot )$. 
Since $\gamma_i$ is relatively simple, by Lemma \ref{PL} 
for any face $\tau \prec \sigma_i$ of 
$\sigma_i$ the constant sheaf 
$\CC_{Y_{\tau}}$ on $Y_{\tau}$ is perverse 
(up to some shift). Moreover by our 
assumption that $f$ has isolated singularities 
at infinity over $b \in 
K_f \setminus 
[f( {\rm Sing} f) \cup \{ f(0) \} ]$, 
the support of $\varphi_{h \circ f_{\tau}} 
( \CC_{Y_{\tau}} )$ is discrete on a neighborhood 
of $T_{\sigma_i} \subset X_{\Sigma_i}$. 
By Lemma \ref{Lemma-2} we 
thus obtain the concentration 
\begin{equation}
H^l \varphi_{h \circ f_{\tau}} 
( \CC_{Y_{\tau}} )_{p_{i,j}} \simeq 0 \ \ \ \ 
(l \not= \dim Y_{\tau}-1 = n- \dim \tau - 1) 
\end{equation}
for any $1 \leq j \leq n_i$ and 
$\tau \prec \sigma_i$. Set 
\begin{equation}
\mu_{i,j, \tau}=
\dim H^{n- \dim \tau - 1} \varphi_{h \circ f_{\tau}} 
( \CC_{Y_{\tau}} )_{p_{i,j}} \geq 0. 
\end{equation}
Then $\mu_{i,j}= (-1)^{n-1} 
\chi \{ \varphi_{h \circ f_i} 
( \F_i )_{p_{i,j}} \} \geq 0$ is expressed 
as a sum of non-negative integers as follows: 
\begin{equation}
\mu_{i,j} = 
\sum_{\tau \prec \sigma_i} 
\mu_{i,j, \tau} \geq 0. 
\end{equation}
Moreover the integer 
$\mu_{i,j, \sigma_i}$ 
is positive by the smoothness of 
$T_{\sigma_i} = 
Y_{\sigma_i}$. 
Consequently we get $E_f(b)>0$. 
This completes the proof. 
\end{proof}

In the generic (Newton non-degenerate) case,  
for any $1 \leq i \leq m$ and 
$1 \leq j \leq n_i$ we can explicitly 
calculate the above integer $\mu_{i,j} \geq 0$ 
by \cite[Theorem 3.4, Corollary 3.6 
and Remark 4.3]{M-T-1} as follows. 
First by multiplying a monomial on $T_{\sigma_i} 
\simeq (\CC^*)^{\dim \gamma_i}$ to $f_i$ we may 
assume that $f_i$ is a regular function on 
$X_i \times \CC^{\dim \gamma_i}$. Next by a 
translation in $\CC^{\dim \gamma_i}$ we reduce 
the problem to the case $p_{i,j}=0 
\in \CC^{\dim \gamma_i}$. 
Then we can apply \cite[Theorem 3.4 
and Corollary 3.6]{M-T-1} 
to $\varphi_{h \circ f_i}( \F_i)_{p_{i,j}} \simeq 
\psi_{h \circ f_i}( \F_i)_{p_{i,j}}$ 
if $f_i:(X_i \times \CC^{\dim \gamma_i}, 0) 
\longrightarrow 
( \CC , 0)$ is Newton non-degenerate at 
$p_{i,j}=0 \in \CC^{\dim \gamma_i}$. 
In this way, even if $\sigma_i$ is not 
simplicial we can express the integer 
$\mu_{i,j} \geq 0$ as an 
alternating sum of the normalized 
volumes of polytopes in $\RR^n_+ \setminus 
\Gamma_+(f)_{i,j}$, where $\Gamma_+(f)_{i,j} 
\subset \RR^n_+$ is the (local) 
Newton polyhedron of $f_i$ at 
$p_{i,j}$. See \cite[Corollary 3.6]{M-T-1} 
for the details. 
We conjecture that it is positive 
in our situation. In the case where 
$n=3$ we have the following stronger result. 

\begin{theorem}\label{SMT-2} 
Assume that $n=3$ and 
$f$ has isolated singularities 
at infinity over $b \in K_f \setminus 
[f( {\rm Sing} f) \cup \{ f(0) \} ]$. 
Then we have $E_f(b)>0$ and hence $b \in B_f$. 
\end{theorem}

\begin{proof}
The proof is similar to that of Theorem 
\ref{SMT-1}. We shall use the notations 
in it. For any $1 \leq i \leq m$ the dimension 
of the atypical face $\gamma_i 
\prec \Gamma_{\infty}(f)$ is $1$ or $2$. 
If $\dim \gamma_i =2$ and $n_i>0$ we 
have $\chi \{ \varphi_{h \circ f_i} 
( \F_i )_{p_{i,j}} \} >0$ for any 
$1 \leq j \leq n_i$ by the result of 
Zaharia \cite{Zaharia}. 
If $\dim \gamma_i =1$ and $n_i>0$ 
the two-dimensional cone $\sigma_i$ 
is simplicial but $\sigma_i \cap \RR_+^3$ 
can be bigger than $\{ 0 \}$. Nevertheless 
we can show the positivity 
 $\chi \{ \varphi_{h \circ f_i} 
( \F_i )_{p_{i,j}} \} >0$ for any 
$1 \leq j \leq n_i$ by calculating 
$\F_i \in \Dbc ( X_{\Sigma_i} )$ 
very explicitly depending on how 
$\sigma_i$ intersects $\RR_+^3$. 
First we consider the case where 
$\dim \sigma_i=2$, 
$\dim \sigma_i \cap \RR_+^3=1$ 
and $\relint ( \sigma_i \cap \RR_+^3) 
\subset \relint ( \sigma_i)$. Then 
for any point $q \in T_{\sigma_i} 
\subset X_{\Sigma_i}$ its fiber of the map 
\begin{equation}
\pi_i|_{\CC^3 \cap X_{\Sigma_i^{\prime}}} : 
\CC^3 \cap X_{\Sigma_i^{\prime}} 
\longrightarrow X_{\Sigma_i}
\end{equation}
is isomorphic to $\CC^*$. 
For its cohomology groups with 
compact support $H^l_c ( \CC^*; \CC )$ 
$(l \in \ZZ )$ we have 
\begin{equation}
H^l_c ( \CC^* ; \CC ) 
 \simeq \begin{cases}
 \CC  & (l=1,2), \\
\ 0 & (l \not= 1,2).  
\end{cases}
\end{equation}
Hence for the point 
$q \in T_{\sigma_i}$ 
we have   
\begin{equation}\label{Erq1} 
 H^l ( \F_i)_q 
 \simeq \begin{cases}
 \CC  & (l=1,2), \\
\ 0 & (l \not= 1,2)  
\end{cases}
\end{equation}
and $\chi ( \F_i) (q)=0$. 
Since the two one-dimensional faces 
$\rho_{i,1}, \rho_{i,2}$ of $\sigma_i$ 
are not contained in $\RR_+^3$ there 
exists also an isomorphism 
\begin{equation}\label{Erq2}
\F_i|_{X_{\Sigma_i} \setminus T_{\sigma_i}} 
\simeq 
( \iota_i)_! \CC_T |_{X_{\Sigma_i} 
\setminus T_{\sigma_i}}
= \CC_T |_{X_{\Sigma_i} 
\setminus T_{\sigma_i}}.  
\end{equation}
It follows from \eqref{Erq1} and 
\eqref{Erq2} we have an equality 
\begin{equation}
\chi \{ \varphi_{h \circ f_i} 
( \F_i )_{p_{i,j}} \} = 
\chi \{ \varphi_{h \circ f_i} 
( ( \iota_i)_! \CC_T )_{p_{i,j}} \} 
= \chi \{ \varphi_{h \circ f_i} 
( \CC_T )_{p_{i,j}} \}
\end{equation}
for any $1 \leq j \leq n_i$. 
Then for any $1 \leq j \leq n_i$ we obtain 
the positivity 
\begin{equation}
\chi \{ \varphi_{h \circ f_i} 
( \F_i )_{p_{i,j}} \} = 
\chi \{ \varphi_{h \circ f_i} 
( \CC_T )_{p_{i,j}} \} >0
\end{equation}
by the proof of Theorem \ref{SMT-1}. 
Next we consider the case where 
$\dim \sigma_i =2$ and 
$\sigma_i \cap \RR_+^3$ is one of 
the two one-dimensional faces 
$\rho_{i,1}, \rho_{i,2}$ of $\sigma_i$. 
We may assume 
that $\sigma_i \cap \RR_+^3=
\rho_{i,1}$. For $1 \leq j \leq 2$ we 
denote by $T_{i,j} \simeq ( \CC^*)^{2}$ 
the $T$-orbit in $X_{\Sigma_i}$ 
associated to $\rho_{i,j} \prec \sigma_i$. 
Then for $Y_{ \{ 2 \} }= 
\overline{T_{i,2}}$ we have an isomorphism 
$\F_i \simeq \CC_{X_{\Sigma_i} \setminus 
Y_{ \{ 2 \} } }$. Since 
$\CC_{X_{\Sigma_i}}$ is a perverse 
sheaf (up to some shift) 
and the two-dimensional 
variety $Y_{ \{ 2 \} }= 
\overline{T_{i,2}}$ is smooth, for any 
$1 \leq j \leq n_i$ we obtain the 
positivity 
\begin{equation}
\chi \{ \varphi_{h \circ f_i} 
( \F_i )_{p_{i,j}} \} = 
\chi \{ \varphi_{h \circ f_i} 
( \CC_{X_{\Sigma_i}} )_{p_{i,j}} \} 
- \chi \{ \varphi_{h \circ f_i} 
( \CC_{ Y_{ \{ 2 \} } } )_{p_{i,j}} \}
\geq - \chi \{ \varphi_{h \circ f_i} 
( \CC_{ Y_{ \{ 2 \} } } )_{p_{i,j}} \}
 >0. 
\end{equation}
Finally, let us treat the case where 
$\dim \sigma_i=  
\dim \sigma_i \cap \RR_+^3=2$.
Since the face $\gamma_i$ is atypical,  
its dual cone $\sigma_i$ is not contained in 
$\RR^3_+$ and hence we have 
$\sigma_i \cap \RR_+^3 \not= 
\sigma_i$ in this case. Assume 
also that $\relint ( \sigma_i \cap \RR_+^3) 
\subset \relint ( \sigma_i)$. Then 
for any point $q \in T_{\sigma_i} 
\subset X_{\Sigma_i}$ its fiber of the map 
\begin{equation}
\pi_i|_{\CC^3 \cap X_{\Sigma_i^{\prime}}} : 
\CC^3 \cap X_{\Sigma_i^{\prime}} 
\longrightarrow X_{\Sigma_i}
\end{equation}
is isomorphic to the singular algebraic curve 
$\{ (x_1,x_2) \in \CC^2 \ | \ 
x_1x_2=0 \} \subset \CC^2$. By 
calculating its Euler characteristic with 
compact support, we obtain  
$\chi ( \F_i) (q)=1$. Moreover we  
have the isomorphism \eqref{Erq2} 
in this case.   
We thus obtain the positivity 
\begin{equation}
\chi \{ \varphi_{h \circ f_i} 
( \F_i )_{p_{i,j}} \} = 
\chi \{ \varphi_{h \circ f_i} 
( \CC_T )_{p_{i,j}} \} 
+ \chi \{ \varphi_{h \circ f_i} 
( \CC_{T_{\sigma_i}} )_{p_{i,j}} \} >0 
\end{equation}
for any $1 \leq j \leq n_i$.
Similarly we can 
prove the non-negativity and the 
positivity also in the 
remaining case.
This completes the proof. 
\end{proof} 

We thus confirm the 
conjecture of \cite{N-Z} 
for $n=3$ in the generic 
case. Similarly, we can improve 
Theorem \ref{SMT-1} as follows. 
In fact, Theorem \ref{SMT-3} below 
extends Theorems \ref{SMT-1} and 
\ref{SMT-2} in a unified manner. Note that 
the condition $\relint ( \gamma_i) 
\subset \Int ( \RR_+^n)$ is equivalent to 
the one $\sigma_i \cap \RR_+^n = \{ 0 \}$ 
for the cone $\sigma_i= \sigma ( \gamma_i) 
\in \Sigma_0$. 

\begin{theorem}\label{SMT-3} 
Assume that $f$ has isolated singularities 
at infinity over $b \in 
K_f \setminus 
[f( {\rm Sing} f) \cup \{ f(0) \} ]$ 
and for any $1 \leq i \leq m$ 
such that $b \in K_i$ 
the set $\sigma_i \cap \RR_+^n$ is a face 
of $\RR_+^n$ of 
dimension $\leq 2$. 
Assume also that there exists $1 \leq i \leq m$ 
such that $b \in K_i$, 
$\gamma_i \prec \Gamma_{\infty}(f)$ 
is relatively simple and moreover in the case 
$\dim \sigma_i \cap \RR_+^n =2$ the number of 
the common edges of $\sigma_i \cap \RR_+^n$ and 
$\sigma_i$ is $\leq 1$. Then we have $E_f(b)>0$ 
and hence $b \in B_f$. 
\end{theorem}

\begin{proof}
The proof is similar to those of Theorems 
\ref{SMT-1} and \ref{SMT-2}.  
We shall use the notations 
in them. In the proof of Theorem \ref{SMT-1} 
we proved for $1 \leq i \leq m$ 
such that $\sigma_i \cap \RR_+^n = \{ 0 \}$ 
(resp. $\sigma_i \cap \RR_+^n = \{ 0 \}$ 
and $\gamma_i$ is relatively simple) we have 
$(-1)^{n-1} \chi \{ \varphi_{h \circ f_i} 
( \F_i )_{p_{i,j}} \} \geq 0$ 
(resp. $>0$) for any 
$1 \leq j \leq n_i$. Let us consider the 
remaining cases where $1 \leq \dim 
\sigma_i \cap \RR_+^n \leq 2$. For a 
face $\tau \prec \sigma_i$ of such 
$\sigma_i$, by taking a reference point 
$q \in T_{\tau} \subset X_{\Sigma_i}$ 
of the $T$-orbit $T_{\tau}$ associated to 
it we set $e( \tau )=\chi ( \F_i )(q)$. 
Then as in the proof of Theorem \ref{SMT-2} 
we can easily show that  
\begin{equation}
e ( \tau )= \begin{cases}
1 & ( \dim \tau \cap \RR_+^n= \dim \tau ), 
\\
\ 0 & ( \dim \tau \cap \RR_+^n< \dim \tau ).  
\end{cases}
\end{equation}
In particular, for the zero-dimensional 
face $\{ 0 \} \prec \sigma_i$ of 
$\sigma_i$ we have $T_{\{ 0 \}}=T$, 
$\F_i|_T \simeq \CC_T$ and $e( \{ 0 \} )
=1$. We thus obtain an equality 
\begin{equation}\label{Erq3} 
(-1)^{n-1} \chi \{ \varphi_{h \circ f_i} 
( \F_i )_{p_{i,j}} \} = (-1)^{n-1}  
\sum_{\tau : e( \tau )=1} 
\chi \{ \varphi_{h \circ f_i} 
( \CC_{T_{\tau}} )_{p_{i,j}} \} 
\end{equation}
for any $1 \leq j \leq n_i$. 
First let us consider the case where 
$\dim \sigma_i \cap \RR_+^n =1$. 
If $\sigma_i \cap \RR_+^n$ is not an edge 
of the cone $\sigma_i$, by \eqref{Erq3} 
we have 
\begin{equation} 
(-1)^{n-1} \chi \{ \varphi_{h \circ f_i} 
( \F_i )_{p_{i,j}} \} = (-1)^{n-1}  
\chi \{ \varphi_{h \circ f_i} 
( \CC_{T} )_{p_{i,j}} \} 
\end{equation}
for any $1 \leq j \leq n_i$.
By the proof of Theorem \ref{SMT-1} this 
integer is non-negative. Moreover it 
is positive if $\gamma_i$ is relatively simple. 
Let $\rho_{i,1}, \rho_{i,2}, \ldots, 
\rho_{i,d_i} \prec \sigma_i$ 
be the edges of $\sigma_i$. For $1 \leq j \leq d_i$ we 
denote by $T_{i,j} \simeq ( \CC^*)^{n-1}$ 
the $T$-orbit in $X_{\Sigma_i}$ 
associated to $\rho_{i,j} \prec \sigma_i$. 
If $\sigma_i \cap \RR_+^n$ is an edge 
$\rho$ of $\sigma_i$, by \eqref{Erq3} 
we can easily see that for the remaining 
edges $\rho_{i,j}$ $(1 \leq j \leq d_i)$ 
of $\sigma_i$ satisfying $\rho_{i,j} 
\not= \rho$ and the hypersurface 
$Z_i:= \cup_{j: \rho_{i,j} 
\not= \rho} \overline{T_{i,j}} 
\subset X_{\Sigma_i}$ defined by them 
there exists an isomorphism $\F_i \simeq 
\CC_{X_{\Sigma_i} \setminus Z_i}$. 
Since the hypersurface complement 
$X_{\Sigma_i} \setminus Z_i$ is an 
affine open subset of $X_{\Sigma_i}$,
$\F_i$ is perverse (up to some shift) 
and we obtain the non-negativity   
\begin{equation} 
(-1)^{n-1} \chi \{ \varphi_{h \circ f_i} 
( \F_i )_{p_{i,j}} \} = (-1)^{n-1}  
\chi \{ \varphi_{h \circ f_i} 
( \CC_{X_{\Sigma_i} \setminus Z_i}
)_{p_{i,j}} \} \geq 0
\end{equation}
for any $1 \leq j \leq n_i$.
Moreover we can rewrite this integer 
as follows:
\begin{equation} 
(-1)^{n-1} \chi \{ \varphi_{h \circ f_i} 
( \F_i )_{p_{i,j}} \} = (-1)^{n-1}  
\sum_{\tau : \rho \not\prec \tau} 
(-1)^{\dim \tau}  
\chi \{ \varphi_{h \circ f_i} 
( \CC_{ \overline{T_{\tau}}} 
)_{p_{i,j}} \}.  
\end{equation}
If $\gamma_i$ is relatively simple, the right 
hand side is a sum of non-negative 
integers and for a facet $\tau$ of 
$\sigma_i$ such that $\rho \not\prec \tau$ 
the closure $\overline{T_{\tau}}$ of 
$T_{\tau}$ is smooth and we have the 
positivity 
\begin{equation} 
(-1)^{n-1+ \dim \tau}  
\chi \{ \varphi_{h \circ f_i} 
( \CC_{ \overline{T_{\tau}}} 
)_{p_{i,j}} \} >0.  
\end{equation}
Finally let us consider the case where 
$\dim \sigma_i \cap \RR_+^n =2$.
Assume that $ ( \sigma_i \cap \RR_+^n) 
\setminus \{ 0 \} 
\subset \relint ( \sigma_i)$. Since the 
case where $\dim \sigma_i = 
\dim \sigma_i \cap \RR_+^n=2$ was already 
treated in the proof of Theorem \ref{SMT-2}, 
here we treat only the case where 
$\dim \sigma_i > 
\dim \sigma_i \cap \RR_+^n=2$.  
Then by \eqref{Erq3} 
we obtain the non-negativity 
\begin{equation} 
(-1)^{n-1} \chi \{ \varphi_{h \circ f_i} 
( \F_i )_{p_{i,j}} \} = (-1)^{n-1}  
\chi \{ \varphi_{h \circ f_i} 
( \CC_{T} )_{p_{i,j}} \} \geq 0 
\end{equation}
for any $1 \leq j \leq n_i$.
Moreover it is positive if 
$\gamma_i$ is relatively simple. 
Similarly we can prove 
the non-negativity and the positivity 
also in the remaining cases. We omit 
the details. This completes the proof. 
\end{proof}

In the case $n=4$ we can also 
partially verify the 
conjecture of \cite{N-Z} as follows. 

\begin{theorem}\label{SMT-4} 
Assume that $n=4$,
$f$ has isolated singularities 
at infinity over $b \in K_f \setminus 
[f( {\rm Sing} f) \cup \{ f(0) \} ]$ 
and for any $1 \leq i \leq m$ such that 
$b \in K_i$ and $\dim \sigma_i= \dim 
\sigma_i \cap \RR^4_+=3$ there exists no 
common edge of $\sigma_i$ and 
$\sigma_i \cap \RR^4_+$. Assume also that 
there exists $1 \leq i \leq m$ such that 
$b \in K_i$ and in the case $\dim \sigma_i=
3$ and  $\dim \sigma_i \cap \RR^4_+=2$ 
the number of the common edges of $\sigma_i$ and 
$\sigma_i \cap \RR^4_+$ is $\leq 1$. 
Then we have $E_f(b)>0$ and hence $b \in B_f$. 
\end{theorem}

\begin{corollary}\label{SMC-5} 
Assume that $n=4$,
$f$ has isolated singularities 
at infinity over $b \in K_f \setminus 
[f( {\rm Sing} f) \cup \{ f(0) \} ]$ 
and for any $1 \leq i \leq m$ such that 
$b \in K_i$ we have $\dim 
\sigma_i \cap \RR^4_+ \leq 1$ or 
$\dim \sigma_i \leq 2$. 
Then we have $E_f(b)>0$ and hence $b \in B_f$. 
\end{corollary}

Since the proof of Theorem \ref{SMT-4} 
is similar to those of Theorems 
\ref{SMT-1}, \ref{SMT-2} and \ref{SMT-3}, 
we omit it here.


\begin{thebibliography}{99}

\bibitem{ALM-1}
Artal Bartolo, E., Luengo, I. and Melle-Hern\'andez, A. 
\emph{Milnor number at infinity, topology and 
Newton boundary of a polynomial function}, 
Mathematische Zeitschrift, 233 (2000): 679-696. 

\bibitem{ALM-2}
Artal Bartolo, E., Luengo, I. and Melle-Hern\'andez, A. 
\emph{On the topology of a generic fibre 
of a polynomial function}, 
Comm. Algebra, 28, No. 4 (2000): 1767-1787.  

\bibitem{Broughton}
Broughton, S. A. \emph{Milnor 
numbers and the topology of polynomial 
hypersurfaces}, Invent. Math., 
92 (1988): 217-241.

\bibitem{C-D-T-T}
Chen, Y., Dias, L. R. G., 
Takeuchi, K. and Tib{\u a}r, M. 
\emph{Invertible polynomial mappings via 
Newton non-degeneracy}, Ann. Inst. Fourier, 
64, No.5 (2014): 1807-1822.  

\bibitem{Dimca}
Dimca, A. Sheaves in topology, 
Universitext, Springer-Verlag, Berlin, 2004.

\bibitem{E-T}
Esterov, A. and Takeuchi, K. 
\emph{Motivic Milnor fibers over complete 
intersection varieties and their virtual 
Betti numbers}, Int. Math. Res. Not., 
Vol. 2012, No. 15 (2012): 3567-3613. 

\bibitem{F}
Fiesler, K. H. 
\emph{Rational intersection cohomology of 
projective toric varieties}, 
J. Reine Angew. Math., 413 (1991): 88-98.

\bibitem{Fulton}
Fulton, W. \emph{Introduction to 
toric varieties}, Princeton University 
Press, 1993.

\bibitem{H-L}
H\`a, H. V. and L\^e, D. T. 
\emph{Sur la topologie des polyn\^omes 
complexes}, Acta Math. Vietnam., 9 (1984): 21-32.

\bibitem{H-N}
H\`a, H. V. and Nguyen, T. T. 
\emph{On the topology of polynomial mappings from 
$\CC^n$ to $\CC^{n-1}$}, Internat. J. Math., 
22 (2011): 435-448.

\bibitem{H-T-T}
Hotta, R., Takeuchi, K. and Tanisaki, T. 
\emph{D-modules, perverse 
sheaves, and representation theory}, 
Birkh{\"a}user Boston, 2008.

\bibitem{K-S}
Kashiwara, M. and Schapira, P. 
Sheaves on manifolds, Springer-Verlag, 1990.

\bibitem{Kushnirenko}
Kouchnirenko, A. G. \emph{Poly\'edres 
de Newton et nombres de Milnor}, 
Invent. Math., 32 (1976): 1-31.

\bibitem{KOS}
Kurdyka, K., Orro, P. and Simon, S. 
\emph{Semilgebraic Sard theorem for generalized critical 
values}, J. Differential Geometry, 
56 (2000): 67-92.

\bibitem{L-S}
Libgober, A. and Sperber, S. 
\emph{On the zeta function of monodromy of a 
polynomial map}, Compositio Math., 
95 (1995): 287-307.

\bibitem{M}
Massey, D. \emph{Hypercohomology 
of Milnor fibres}, 
Topology, 35, No. 4 (1996): 969-1003.

\bibitem{M-T-1}
Matsui, Y. and Takeuchi, K. 
\emph{Milnor fibers over singular toric varieties 
and nearby cycle sheaves}, 
Tohoku Math. J., 63 (2011): 113-136.

\bibitem{M-T-2}
Matsui, Y. and Takeuchi, K. 
M\emph{onodromy zeta functions at infinity, Newton 
polyhedra and constructible sheaves}, 
Mathematische Zeitschrift, 
268 (2011): 409-439.

\bibitem{M-T-3}
Matsui, Y. and Takeuchi, K. 
\emph{A geometric degree formula for 
$A$-discriminants and Euler 
obstructions of toric varieties}, Adv. 
in Math., 226 (2011): 2040-2064.

\bibitem{M-T-4}
Matsui, Y. and Takeuchi, K. \emph{Monodromy 
at infinity of polynomial maps and 
Newton polyhedra, with Appendix by C. Sabbah}, 
Int. Math. Res. Not., 
Vol. 2013, No. 8 (2013): 1691-1746. 

\bibitem{Milnor}
Milnor, J. Singular 
points of complex hypersurfaces, Princeton 
University Press, 1968.

\bibitem{N-Z}
N{\'e}methi, A. and Zaharia A. 
\emph{On the bifurcation set of a polynomial 
function and Newton boundary}, 
Publ. Res. Inst. Math. Sci., 
26 (1990): 681-689.

\bibitem{Nguyen} 
Nguyen, T. T., \emph{Bifurcation set, $M$-tameness, 
asymptotic critical values and Newton 
polyhedrons}, Kodai Math. J., 
Vol. 36, No. 1 (2013): 77-90. 

\bibitem{Oda}
Oda, T. Convex bodies and 
algebraic geometry. An introduction to the 
theory of toric varieties, Springer-Verlag, 1988.

\bibitem{Oka}
Oka, M. Non-degenerate complete 
intersection singularity, Hermann, 
Paris (1997).

\bibitem{P} 
Parusinski, A., \emph{On the bifurcation set of complex 
polynomial with isolated singularities at infinity}, 
Compositio Math., 
Vol. 97, No. 3 (1995): 369-384. 

\bibitem{Ra}
Rabier, P.J. \emph{Ehresmann's fibrations and Palais-Smale 
conditions for morphisms of Finsler manifolds}, 
Ann. of Math., 146 (1997): 647-691.

\bibitem{S-T-1}
Siersma, D. and Tib{\u a}r, M. 
\emph{Singularities at infinity and their 
vanishing cycles}, Duke Math. J., 
80 (1995): 771-783.

\bibitem{S}
Suzuki, M. \emph{Propri\'et\'es topologiques des polyn\^omes de 
deux variables complexes, et automorphismes alg\'ebriques 
de l'espace $\CC^2$}, 
J. Math. Soc. Japan, 26 (1974): 241-257.

\bibitem{T-T}
Takeuchi, K. and Tib{\u a}r, \emph{M. Monodromies 
at infinity of non-tame polynomials}, 
Bull. Soc. Math. France, 
144, no.3 (2016): 477-506.

\bibitem{Tibar}
Tib{\u a}r, M. \emph{Topology at infinity of polynomial 
mappings and Thom regularity condition}, 
Compositio Math., 
111, no.1 (1998), 89-109.

\bibitem{T-2}
Tib{\u a}r, M. 
\emph{Asymptotic equisingularity and topology 
of complex hypersurfaces}, Int. Math. Res. Not., 
Vol. 1998, No. 18 (1998): 979-990. 

\bibitem{Tibar-book}
Tib{\u a}r, M. 
Polynomials and vanishing cycles, 
Cambridge University Press, 2007.

\bibitem{Varchenko}
Varchenko, A. N. \emph{Zeta-function 
of monodromy and Newton's diagram}, 
Invent. Math., 37 (1976): 253-262.

\bibitem{Zaharia}
Zaharia A. \emph{On the bifurcation set of a polynomial 
function and Newton boundary II}, 
Kodai Math. J., 
19 (1996): 218-233.


\end{thebibliography}
\end{document}